\documentclass[12pt,times,a4paper]{amsart}
\usepackage{amssymb,amsmath,amsfonts,latexsym,color,amsthm,tikz}
\usepackage{mathtools}
\usepackage{times}
\usepackage{calc}
\usepackage[T1]{fontenc}
\usepackage[latin2]{inputenc}
\usepackage{graphicx}
\allowdisplaybreaks
\usepackage[all]{xy}
\usepackage{enumitem}
\def\id{\mathop{\rm id}\nolimits}

\def\0D{\Delta^{(0)}}
\def\1D{\Delta^{(1)}}

\newcommand{\longto}{\longrightarrow}

\newtheorem{theorem}{Theorem}[section]
\newtheorem{remark}[theorem]{Remark}
\newtheorem{proposition}[theorem]{Proposition}
\newtheorem{lemma}[theorem]{Lemma}
\newtheorem{corollary}[theorem]{Corollary}

\newtheorem{example}[theorem]{Example}
\newtheorem{definition}[theorem]{Definition}

\def\build#1_#2^#3{\mathrel{
\mathop{\kern 0pt#1}\limits_{#2}^{#3}}}

\numberwithin{equation}{section}

\def\part{\partial}

\def\text{\hbox}

\def\build#1_#2^#3{\mathrel{
\mathop{\kern 0pt#1}\limits_{#2}^{#3}}}

\numberwithin{equation}{section}

\newcommand{\comment}[1]{\relax}

\title{\vspace*{-15mm} The covariant functoriality of graph algebras}
\author[P. M.~Hajac]{Piotr M.~Hajac}
\address[P. M.~Hajac]{Instytut Matematyczny, Polska Akademia Nauk, ul. \'Sniadeckich 8, Warszawa, 00-656 Poland}
\email{pmh@impan.pl }

\author[M.~Tobolski]{Mariusz Tobolski}
\address[M.~Tobolski]{Instytut Matematyczny, Uniwersytet Wroc\l{}awski, pl. Grunwaldzki~2, Wroc\l{}aw, 50-384 Poland}
\email{mariusz.tobolski@math.uni.wroc.pl}

\begin{document}
\baselineskip15pt
\begin{abstract}
In the standard category of directed graphs, graph morphisms map edges to edges. By allowing graph morphisms 
to map edges to finite paths (path homomorphisms of graphs), we obtain  an ambient category in which we determine subcategories 
enjoying covariant
functors to categories of algebras given by constructions of path algebras, Cohn path algebras, and Leavitt path algebras,
respectively. Thus we obtain new tools to unravel homomorphisms between Leavitt path algebras and between graph C*-algebras.
In particular, a graph-algebraic presentation of the inclusion of the C*-algebra of a quantum real projective plane into the Toeplitz
algebra allows us to determine a quantum CW-complex structure of the former. It comes as a mixed-pullback theorem where
two $*$-homomorphisms are covariantly induced from path homomorphisms of graphs and the remaining two are contravariantly 
induced by admissible inclusions of graphs. As a main result and an application of new covariant-induction tools, 
we prove such a mixed-pullback
theorem for arbitrary graphs whose all vertex-simple loops have exits, which  substantially enlarges the scope of examples coming 
from noncommutative topology. 
\end{abstract}
\maketitle

\section{Introduction}
\noindent
The standard category of directed graphs and graph homomorphism found spectacular application in solving problems concerning 
finitely generated free groups~\cite{s-jr83}. Herein, prompted by interesting examples and the flexibility afforded by the Cuntz--Krieger relations~\cite{ck-80}, 
we consider a much larger category of directed
graphs whose morphisms are path homomorphisms of graphs. The main idea is that path homomorphisms need not preserve the lengths 
of paths, which is well suited to study $*$-homomorphisms between graph C*-algebras \cite{bprs-00,flr-00} that are not gauge equivariant.  

To systematically understand different layers of the covariant induction~\cite{s-j02}, first we observe that the construction of a path algebra defines a covariant functor
on the subcategory given by path homomorphisms that are injective on vertices. Then we propose a new monotonicity condition for
such path homomorphisms (generalizing \cite[Lemma~3.3(1)]{cht21}), and prove that it defines a further restricted subcategory which yields a covariant functor 
via the construction
of a Cohn path algebra~\cite{c-pm66}. Finally, we unravel a general regularity condition, vastly generalizing its earlier incarnation~\cite[Section~2.3]{kr-g09} (cf.~\cite[Definition~1.6.2]{aasm17}),
custom-designed to handle the second-type Cuntz--Krieger relation, and thus arrive at the subcategory bringing about a~covariant functor via
the contruction of a Leavitt path algebra and its enveloping graph C*-algebra. 

As a key application, we use the aforementioned covariant functor into graph C*-algebras to prove a mixed-pullback theorem (Theorem~\ref{main})
substantially generalizing the mixed-pullback theorem~\cite[Theorem~3.4]{cht21}. In particular, we replace the acyclicity assumption in the latter by ``all vertex-simple 
loops have exits'' in the former (cf.~ \cite[Theorem~2]{flr-00}). We instantiate our theorem by the graph of a quantum projective plane \cite{h-pm96,hrs-03,hs02}, which has a vertex-simple loop
with an exit. The thus obtained pullback
structure of its C*-algebra can be interpreted as a quantum CW-complex structure~\cite{dhmsz-20}.

\section{Preliminaries}

\subsection{The standard category of graphs}

A (directed) {\em graph} is a quadruple $E:=(E^0,E^1,s_E,t_E)$, where:
\begin{itemize}
\item $E^0$ is the set of {\em vertices},
\item $E^1$ is the set of {\em edges},
\item $E^1\overset{s_E}{\to}E^0$ is the {\em source} map assigning to each edge its beginning,
\item $E^1\overset{t_E}{\to}E^0$ is the {\em target} (range) map assigning to each edge its end.
\end{itemize}

Let $v$ be a vertex in a graph $E$. It is called a {\em sink} iff $s^{-1}_E(v)=\emptyset$,  it is called a {\em source} iff~$t^{-1}_E
(v)=\emptyset$, and it is called {\em regular} iff it is not a sink and $|s_E^{-1}(v)|<\infty$. The subset of regular vertices of a graph 
$E$ is denoted by ${\rm reg}(E)$. A {\em finite path} $p$ in $E$ is a vertex or a finite collection $e_1\ldots e_n$ of edges satisfying 
\begin{equation}
t_E(e_1)=s_E(e_2),\quad t_E(e_2)=s_E(e_3),\quad \ldots,\quad t_E(e_{n-1})=s_E(e_n).
\end{equation} 
We denote the set of all finite paths in $E$ by~$FP(E)$. The beginning $s_{PE}(p)$ of $p$ is $s_E(e_1)$ and the end $t_{PE}(p)$ of $p$ is $t_E(e_n)$. The 
beginning and the end of a vertex is the vertex itself. We thus extend the source and target maps to maps $
s_{PE},\,t_{PE}\colon FP(E)\rightarrow E^0$.
 Vertices are considered as finite paths of length~$0$. Every edge is a path of 
length~$1$. The {\em length} of a finite path that is not a vertex is the size of the tuple. The set of all paths in $E$ of length $n$
is denoted by~$FP_n(E)$. 
A {\em loop} is a positive-length path that begins and  ends at the same vertex.

\begin{definition}
A {\em homomorphism} from a graph $E:=(E^0,E^1,s_E,t_E)$ to a graph\linebreak $F:=(F^0,F^1,s_F,t_F)$
is a pair of maps 
\begin{equation*}
(f^0:E^0\to F^0,f^1:E^1\to F^1)
\end{equation*} 
satisfying the conditions:
\begin{equation*} 
s_F\circ f^1=f^0\circ s_E\,,\qquad t_F\circ f^1=f^0\circ t_E\,.
\end{equation*}
We call $(f_0,f_1)$ {\em injective} or {\em surjective} if both $f_0$ and $f_1$ are injective or surjective, respectively.
\end{definition}
\noindent One can easily check that graphs with homomorphisms of graphs form a category. We refer to this category as the 
\emph{standard category of graphs}.

\subsection{Different types of graph algebras}
Let $E$ be a graph and $R_E$ be the set of all subsets of~$FP(E)$. Equiping the latter with
the set union $\cup$ as addition and the concatenation of paths as multiplication renders $R_E$ a unital
semiring. Here the neutral element for addition is $\emptyset$ and the neutral element for multiplication
is~$E^0$. Restricting this structure to the set $R^f_E$ of all finite subsets of~$FP(E)$ yields a~subsemiring
of~$R_E$. The semiring $R^f_E$ is unital if and only if $|E^0|<\infty$, in which case the unit is~$E^0$.

Next, let $E$ be a non-empty graph and let $k$ be any field. 
The following construction can be thought of as a linearization of~$R^f_E$. 
Consider the vector space
\begin{equation}
kE:=\{f\in{\rm Map}(FP(E),k)~|~f(p)\neq 0\text{ for finitely many } p\in FP(E)\},
\end{equation}
where the addition and scalar multiplication are pointwise. Then the set of functions $\{\chi_p\}_{p\in FP(E)}$
given by
\begin{equation}
\chi_p(q):=\delta_{p,q}:=\begin{cases}1 &\text{for }p=q\\0 & \text{otherwise}\end{cases}
\end{equation}
is a linear basis of $kE$. By checking the associativity, one can prove that the formulas
\begin{equation}
m:kE\times kE\longto kE,\qquad m(\chi_p,\chi_q):=\begin{cases}\chi_{pq} & \text{if }t(p)=s(q)\\0 & \text{otherwise}\end{cases}.
\end{equation}
define a multiplication on~$kE$. 
\begin{definition}{{\rm (\cite[Definition~II.1.2]{ass-06})}}
Let $E$ be a non-empty graph. 
The above constructed algebra $(kE,+,0,m)$ is called the {\em path algebra} of~$E$.
If $E=\emptyset$, then $kE:=0$.
\end{definition}

To define Cohn and Leavitt path algebras, we need extended graphs.
The {\em extended graph} $\bar{E}:=(\bar{E}^0, \bar{E}^1,s_{\bar{E}},t_{\bar{E}})$
of the graph $E$ is given as follows:
\begin{gather}
\bar{E}^0:=E^0,\quad \bar{E}^1:=E^1\sqcup (E^1)^*,\quad (E^1)^*:=\{e^*~|~e\in E^1\},
\nonumber\\
\forall\; e\in E^1:\quad s_{\bar{E}}(e):=s_E(e),\quad t_{\bar{E}}(e):=t_E(e),
\nonumber\\
\forall\; e^*\in (E^1)^*:\quad s_{\bar{E}}(e^*):=t_E(e),\quad t_{\bar{E}}(e^*):=s_E(e).
\end{gather}
Given a path $p=e_1e_2\ldots e_n\in FP(E)$, we put $p^*:=e_n^*\ldots e_1^*$ for the corresponding path in~$FP(\overline{E})$. 
\begin{definition}{\rm(\cite[Definition~1.5.9]{aasm17})}\label{relcohndef}
Let $E$ be a graph, $k$ be a field, and $X$ be a subset of the set $\mathrm{reg}(E)$ 
of all regular vertices in~$E$. The {\em relative Cohn path algebra} $C_k^X(E)$ of $E$
is the path algebra $k\bar{E}$ of the extended graph $\bar{E}$ divided by the ideal generated
by the union of the following sets:
\vspace*{-2mm}
\begin{enumerate}
\item[{\rm (1)}] $\{\chi_{e^*}\chi_f-\delta_{e,f}\chi_{t_E(e)}~|~e,f\in E^1\}$,
\item[{\rm (2)}] $\{\sum_{e\in s^{-1}_E(v)}\chi_{e}\chi_{e^*}-\chi_v~|~v\in X\}$.
\end{enumerate}
\vspace*{-2mm}
The relative Cohn algebra $C_k^\emptyset(E)=:C_k(E)$ is called the {\em Cohn path algebra}
 of~$E$, and the relative Cohn algebra $C_k^{\mathrm{reg}(E)}(E)=:L_k(E)$
 is called the {\em Leavitt path algebra} of~$E$.
\end{definition}
\noindent 
The algebra $C_k^X(E)$ is isomorphic to the universal algebra generated by the elements 
$\chi_v$, $\chi_e$, $\chi_{e^*}$, $v\in E^0$, $e\in E^1$, subject to the above relations (1) and (2) and
 the standard path-algebraic relations $\chi_v\chi_w=\delta_{v,w}\chi_v$, $v,w\in E^0$, and $\chi_{s_E(e)}\chi_e=\chi_e=\chi_e\chi_{t_E(e)}$, $e\in E^1$.

In the case $k=\mathbb{C}$, one defines an anti-linear involution on $L_\mathbb{C}(E)$  by the formulas
\begin{equation}
([\chi_v])^*:=[\chi_v],\quad ([\chi_e])^*:=[\chi_{e^*}],\quad ([\chi_{e^*}])^*:=[\chi_e],\quad v\in E^0,\quad e\in E^1,
\end{equation}
to render  $L_\mathbb{C}(E)$ a complex $*$-algebra.
This brings us to the following key definition:
\begin{definition}{{\rm (\cite[Definition~5.2.1]{aasm17})}}\label{graphc*}
Let $E$ be a graph. The {\em graph C*-algebra} $C^*(E)$ of $E$ is 
the universal enveloping C*-algebra of the complex $*$-algebra $L_\mathbb{C}(E)$.
\end{definition}
\noindent 
Definition~\ref{graphc*} is equivalent to~\cite[Definition~1]{flr-00} defining $C^*(E)$ as the universal C*-algebra 
generated by mutually orthogonal projections $P_v$, $v\in E^0$, and partial isometries $S_e$, $e\in E^1$, with mutually 
orthogonal ranges, satisfying
\begin{enumerate}
\item $S^*_eS_e=P_{t_E(e)}$ for all $e\in E^1$,
\item $P_v=\sum_{e\in s^{-1}_E(v)}S_eS_e^*$ for all $v\in {\rm reg}(E)$,
\item $S_eS_e^*\leq P_{s_E(e)}$ for all $e\in E^1$.
\end{enumerate}
In what follows, we will need the notaion $S_p:=S_{e_1}S_{e_2}\ldots S_{e_n}$,
 where $p=e_1e_2\ldots e_n$ is a finite path in~$E$, and $S_v:=P_v$, where $v\in E^0$. 

\section{Subcategories in the category of graphs and path morphisms of graphs}
\noindent
Let us now consider a more general category of graphs. 
\begin{definition}
Let $E$ and $F$ be graphs.
A {\em path homomorphism} from $E$ to $F$  is a map\linebreak \mbox{$f\colon FP(E)\to FP(F)$} satisfying:
\begin{enumerate}
\item
$f(E^0)\subseteq F^0$,
\item
$s_{PF}\circ f=f\circ s_{PE}\,,\quad t_{PF}\circ f=f\circ t_{PE}\,$,
\item
$\forall\;p,q\in FP(E)\text{ such that } t_{PE}(p)=s_{PE}(q)\colon f(pq)=f(p)f(q)$.
\end{enumerate}
\end{definition}
\noindent If $(f^0,f^1)\colon E\to F$ is a homomorphism of graphs, then $f\colon FP(E)\to FP(F)$ defined by
\begin{gather}
\forall\; v\in E^0\colon f(v):=f^0(v),\quad
\forall\; e\in E^1\colon f(e):=f^1(e),\\ \nonumber
\forall\; e_1\ldots e_n\in FP(E)\colon f(e_1\ldots e_n):= f^1(e_1)\ldots f^1(e_n)\in FP(F),
\end{gather}
is a path homomorphism of graphs. The thus induced path homomorphism of graphs preserves the length of
paths.  Vice versa, if a path homomorphism of graphs preserves the length of
paths, then it is induced by a homomorphism of graphs. Therefore, if a path homomorphism of graphs maps
edges to paths longer than one or to vertices, then it 
does not come from a homomorphism of graphs. Note also that if we treat a graph as a small category, where vertices are objects and 
finite paths are morphisms between them, then  path homomorphisms are  functors between these categories.
Therefore, it is clear that path homomorphisms yield a category:
\begin{definition}
We denote the category of graphs and path homomorphisms by $PG$ and its subcategory of graphs and path homomorphisms which 
are injective  on the set of vertices by~$IPG$. The subcategory of $IPG$ obtained by restricting objects to graphs with finitely many 
vertices and morphisms to path homomorphisms that are bijective on vertices is denoted by~$BPG$.
\end{definition}

To deal with the condition Definition~\ref{relcohndef}(1), we introduce the following relation on the set of finite paths:
\begin{equation}\label{popaths}
\alpha\preceq\beta\iff\exists\;\gamma\in FP(E)\colon \beta=\alpha\gamma.
\end{equation}
Note that $\preceq$ is a partial order. Indeed, reflexivity holds because $\alpha=\alpha t_E(\alpha)$,
transitivity holds because
\begin{equation}
(\beta=\alpha\gamma\;\text{and}\;\delta=\beta\gamma')\quad\Longrightarrow\quad
\delta=\alpha\gamma\gamma',
\end{equation}
and antisymmetry holds because
\begin{equation*}
(\beta=\alpha\gamma\;\text{and}\;\alpha=\beta\gamma')\quad\Rightarrow\quad
\beta=\beta\gamma'\gamma\quad\Rightarrow\quad \gamma'=t_{PE}(\beta)=\gamma
\quad\Rightarrow\quad \alpha=\beta t_{PE}(\beta)=\beta .
\end{equation*}
\begin{lemma}\label{mlemma}
Restricting morphisms in the category $IPG$ to the morphisms satisfying the {\em monotonicity} condition
\begin{equation}\label{mcondition}
f(e)\preceq f(e')\quad\Longrightarrow\quad e=e',\quad
\end{equation}
where $e$ and $e'$ are edges, yields a subcategory of~$IPG$.
\end{lemma}
\begin{proof}
For starters, for each graph $E$, \eqref{mcondition} is satified by $f=\id_E$. Assume now that both
morphisms $f\colon E\to F$ and $g\colon F\to G$ satisfy~\eqref{mcondition} and $g(f(e))\preceq g(f(e'))$. 
If $f(e)$ is a vertex,
then $g(f(e))$ is also a~vertex, so
\begin{equation}
g(f(e))=s_{PG}\big(g(f(e'))\big)=g\big(s_{PF}(f(e'))\big).
\end{equation}
Hence,  $f(e)=s_{PF}(f(e'))$ by the vertex injectivity of~$g$. Consequently,
 $f(e)\preceq f(e')$, which implies~\mbox{$e=e'$}.
Much in the same way, if $f(e')$ is a vertex, then also $g(f(e'))$ is a vertex, so 
\begin{equation}
g(f(e'))=g(f(e))=s_{PG}\big(g(f(e))\big)=g\big(s_{PF}(f(e))\big).
\end{equation}
Hence, $f(e')=s_{PF}(f(e))$ by the vertex injectivity of~$g$. Consequently, 
$f(e')\preceq f(e)$, 
which implies~\mbox{$e=e'$}.

Assume now that 
$f(e)=e_1\ldots e_m$ and $f(e')=e'_1\ldots e'_n\,$,
where $e_1,\ldots, e_m,e'_1,\ldots, e'_n$ are edges.
Then
$g(e_1)\ldots g(e_m)\preceq g(e'_1)\ldots g(e'_n)$ implies that $g(e_1)\preceq g(e'_1)$
or $g(e'_1)\preceq g(e_1)$, whence $e_1=e'_1$. It follows that 
$g(e_2)\ldots g(e_m)\preceq g(e'_2)\ldots g(e'_n)$. We can iterate this argument to conclude that
$e_i=e'_i$ for all $i\in\{1,\ldots,\mathrm{min}\{m,n\}\}$. Therefore, $f(e)\preceq f(e')$ or\linebreak
$f(e')\preceq f(e)$,
so~\mbox{$e=e'$}. 
Summarizing, we have shown that
\begin{equation}
(f\circ g)(e)\preceq (f\circ g)(e')\quad\Rightarrow\quad e=e',
\end{equation}
so \eqref{mcondition} determines a subcategory.
\end{proof}

\begin{definition}
The subcategory of the category $IPG$ given by Lemma~\ref{mlemma} is denoted by $MIPG$ and its subcategory given by restricting 
its objects and morphisms to $BPG$ is denoted by~$MBPG$.
\end{definition}

Finally, to deal with the condition Definition~\ref{relcohndef}(2), we further restrict our morphisms.
\begin{definition}\label{regularity}
Let $f\colon FP(E)\to FP(F)$ be a morphism in the category $PG$ and let 
\begin{equation*}
{\rm reg}_0(E):=\{v\in{\rm reg}(E)~|~s_E^{-1}(v)=\{e\}\text{ and } t_E(e)=v\}
\end{equation*} 
denote the set of {\em 0-regular vertices}. We call $f$ \emph{regular} whenever the following conditions hold:
\vspace*{-2mm}\begin{enumerate}
\item
For any $v\in\mathrm{reg}(E)\setminus {\rm reg}_0(E)$, we require that:
\begin{enumerate}
\item
$f$ restricted to $s_E^{-1}(v)$ be injective;
\item
$p\in f(s_E^{-1}(v))$ if and only if
\begin{enumerate}
\item
$\exists\;n\in\mathbb{N}\setminus\{0\}\colon p=e_1\ldots e_n,\, e_i\in F^1 \text{ for all } i$,
\item
$pq\in f(s_E^{-1}(v))\,\Rightarrow\, q=t_{PE}(p)$,
\item
$\forall i\in\{1,\ldots,n\}, e\in s_F^{-1}(s_F(e_i))\,\exists\,r\in FP(F):e_1\ldots e_{i-1}er\in f(s_E^{-1}(v))$.
\end{enumerate}
\end{enumerate}
\item 
For any $v\in {\rm reg}_0(E)$, either the above condition holds or $f(s^{-1}_E(v))=f(v)$.
\end{enumerate}
\end{definition}
\noindent
Note that all vertices in the paths belonging to $f(s^{-1}_E(v))$, possibly except for their endpoints, are regular. Also, observe that
the set
$f(s^{-1}_E(v))$ can be viewed as constructed in the following way: we take  $x\in s_F^{-1}(f(v))$ 
and either
set it aside as a length-one element of $f(s^{-1}_E(v))$, or extend it by all edges emitted from $t_F(x)$. Any thus 
obtained path of length two, we either set aside as an element of 
$f(s^{-1}_E(v))$, or extend it by all edges emitted from its end. Then we iterate this procedure until we obtain the
set~$f(s^{-1}_E(v))$. 

\begin{example}
{\rm Let $f\colon E\to F$ be a path homomorphism of graphs such that the restriction
\begin{equation*}
s^{-1}_E(v)\longrightarrow s^{-1}_F(f(v))
\end{equation*}
is bijective for all $v\in\mathrm{reg}(E)$. Then $f$ is regular. Observe that the above condition was considered in~\cite[Section~2.3]{kr-g09}.}
\end{example}
\begin{example}
{\rm Let 
$$
E:=\quad
\begin{tikzpicture}[auto,swap]
\tikzstyle{vertex}=[circle,fill=black,minimum size=3pt,inner sep=0pt]
\tikzstyle{edge}=[draw,->]
\tikzstyle{cycle1}=[draw,->,out=130, in=50, loop, distance=40pt]
\tikzstyle{cycle2}=[draw,->,out=130, in=50, loop, distance=70pt]
   
\node[vertex,label=below:$v$] (0) at (0,0) {};

\path (0) edge[cycle1] node[above] {$e$} (0);

\end{tikzpicture}\text{ and }\quad
F:=\quad
\begin{tikzpicture}[auto,swap]
\tikzstyle{vertex}=[circle,fill=black,minimum size=3pt,inner sep=0pt]
\tikzstyle{edge}=[draw,->]
\tikzstyle{cycle1}=[draw,->,out=130, in=50, loop, distance=40pt]
\tikzstyle{cycle2}=[draw,->,out=130, in=50, loop, distance=70pt]
   
\node[vertex,label=below:$v$] (0) at (0,0) {};


\end{tikzpicture}.
$$
Then $f\colon FP(E)\to FP(F)$ must be the constant map, and it is regular by the condition (2) of Definition~\ref{regularity}.}
\end{example}
\begin{example}
{\rm Let 
$$
E:=\begin{tikzpicture}[auto,swap]
\tikzstyle{vertex}=[circle,fill=black,minimum size=3pt,inner sep=0pt]
\tikzstyle{edge}=[draw,->]
   
\node[vertex,label=below:$v$] (0) at (0,0) {};
\node[vertex,label=below:$w$] (3) at (1,0) {}; 

\path (0) edge[edge] node[above] {$e$} (3);

\end{tikzpicture}\quad \text{and} \quad
F:=\begin{tikzpicture}[auto,swap]
\tikzstyle{vertex}=[circle,fill=black,minimum size=3pt,inner sep=0pt]
\tikzstyle{edge}=[draw,->]
   
\node[vertex,label=below:$a$] (0) at (0,0) {};
\node[vertex,label=below:$b$] (1) at (1,0) {}; 
\node[vertex,label=below:$c$] (2) at (2,0) {}; 

\path (0) edge[edge] node[above] {$x$} (1);
\path (1) edge[edge] node[above] {$y$} (2);

\end{tikzpicture}.
$$
Then mapping $v$ to $a$, $w$ to $c$, and $e$ to $xy$ defines a regular path homomorphism of graphs.}
\end{example}
\begin{example}
{\rm Let \begin{equation*}E:=\begin{tikzpicture}[auto,swap]
\tikzstyle{vertex}=[circle,fill=black,minimum size=3pt,inner sep=0pt]
\tikzstyle{edge}=[draw,->]
\tikzstyle{cycle1}=[draw,->,out=130, in=50, loop, distance=40pt]
\tikzstyle{cycle2}=[draw,->,out=100, in=30, loop, distance=40pt]
   
\node[vertex,label=below:$v$] (0) at (0,0) {};
\node[vertex,label=below:$w_1$] (1) at (-1,-1) {};
\node[vertex,label=below:$w_2$] (2) at (1,-1) {};

\path (0) edge[edge] node[above left] {$f_1$} (1);
\path (0) edge[edge] node[above right] {$f_2$} (2);

\end{tikzpicture}\quad \text{and} \quad F:=
\begin{tikzpicture}[auto,swap]
\tikzstyle{vertex}=[circle,fill=black,minimum size=3pt,inner sep=0pt]
\tikzstyle{edge}=[draw,->]
\tikzstyle{cycle1}=[draw,->,out=130, in=50, loop, distance=40pt]
\tikzstyle{cycle2}=[draw,->,out=100, in=30, loop, distance=40pt]
   
\node[vertex,label=below:$v$] (0) at (0,0) {};
\node[vertex,label=below:$w_1$] (1) at (-1,-1) {};
\node[vertex,label=below:$w_2$] (2) at (1,-1) {};

\path (0) edge[cycle1] node[above] {$u$} (0);
\path (0) edge[edge] node[above left] {$f_1$} (1);
\path (0) edge[edge] node[above right] {$f_2$} (2);

\end{tikzpicture}.
\end{equation*}
Then there is no regular path homomorphism from $E$ to~$F$.}
\end{example}

\begin{example}
{\rm Let \begin{equation*}
E:=\quad\begin{tikzpicture}[auto,swap]
\tikzstyle{vertex}=[circle,fill=black,minimum size=3pt,inner sep=0pt]
\tikzstyle{edge}=[draw,->]
\tikzstyle{cycle1}=[draw,->,out=130, in=50, loop, distance=40pt]
\tikzstyle{cycle2}=[draw,->,out=100, in=30, loop, distance=40pt]
   
\node[vertex,label=below:$u$] (-1) at (-1,0) {};
\node[vertex,label=below:$v$] (0) at (0,0) {};
\node[vertex,label=below:$w$] (1) at (1,0) {};

\path (0) edge[edge] node[above] {$e_0$} (-1);
\path (0) edge[edge, bend right] node[below] {$e_2$} (1);
\path (0) edge[edge,bend left] node[above] {$e_1$} (1);

\end{tikzpicture}\quad\text{and}\quad
F:=\quad\begin{tikzpicture}[auto,swap]
\tikzstyle{vertex}=[circle,fill=black,minimum size=3pt,inner sep=0pt]
\tikzstyle{edge}=[draw,->]
\tikzstyle{cycle1}=[draw,->,out=130, in=50, loop, distance=40pt]
\tikzstyle{cycle2}=[draw,->,out=100, in=30, loop, distance=40pt]

\node[vertex,label=below:$v$] (0) at (0,0) {};
\node[vertex,label=below:] (1) at (1,0) {};
\node[vertex,label=below:$u$] (2) at (2,1) {};
\node[vertex,label=below:$w$] (3) at (2,-1) {};

\path (0) edge[edge, bend right] node[below] {$x_2$} (1);
\path (0) edge[edge,bend left] node[above] {$x_1$} (1);
\path (1) edge[edge] node[above] {$y_1$} (2);
\path (1) edge[edge] node[above] {$y_2$} (3);

\end{tikzpicture}.
\end{equation*}
Then mapping $e_0$ to $x_1y_1$, $e_1$ to $x_1y_2$, and $e_2$ to $x_2y_2$ defines a  path homomorphism
of graphs that is not regular because we are missing~$x_2y_1$.}
\end{example}

\begin{lemma}\label{rlemma}
Restricting morphisms in the category $MIPG$ to the morphisms 
 satisfying the regularity condition of Definition~\ref{regularity}
 yields a subcategory of~$MIPG$. 
\end{lemma}
\begin{proof}
To begin with, it is clear that, for each graph $E$, $f=\id_E$ satisfies the regularity condition. Assume now that 
$f\colon E\to F$ and $g\colon F\to G$ are morphisms in~$MIPG$. Then, by Lemma~\ref{mlemma}, 
$g\circ f$
is injective on edges. 
Next, let us observe that regular vertices that are not 0-regular must be mapped to regular vertices that are not 0-regular by
any morphism in $MIPG$ that satisfies the regularity condition.
Indeed, if $v\in \mathrm{reg}(E)\setminus\mathrm{reg}_0(E)$, 
then $f(v)\in \mathrm{reg}(E)$. Suppose now that $f(v)\in \mathrm{reg}_0(E)$.
Then, since $v$ emits more than one edge, or emits an edge that is not a~loop, and $f(v)$ 
emits only one edge that is a loop, we immediately obtain a 
contradiction with the regularity condition or the injectivity-on-vertices condition.

Now, let us consider the composition of morphisms restricted to regular vertices that are not 0-regular.
From the regularity of $f$ and the monotonicity of $g$, we infer that 
\begin{equation}
g\circ f\colon s_E^{-1}(v)\to g(f(s^{-1}_E(v)))
\end{equation}
 is bijective for any such a vertex~$v$. 
Indeed, arguing as in the proof of Lemma~\ref{mlemma}, 
we see that the only different
paths which $g$ can glue are paths such that one extends the other, and there are no such pairs of different paths in $f(s^{-1}_E(v))$.
Furthermore, since $f(v)$ is regular but not $0$-regular, and $g$ is regular, all elements of $g(f(s_E^{-1}(v)))$ are positive-length paths. Better still, the implication $p$, $pq\in g(f(s^{-1}_E(v)))$ $\Rightarrow$ $q=t_{PE}(p)$ follows immediately from the monotonicity of $g\circ f$ guaranteed by Lemma~\ref{mlemma}. Next, assume that $e_1\ldots e_n\in g(f(s^{-1}_E(v)))$ and $e\in s^{-1}_G(s_G(e_i))$ for some $i\in\{1,\ldots, n\}$. Then $e_1\ldots e_i=g(x_1\ldots x_m)$, where $x_j\in F^1$, $1\leq j\leq m$. As $x_1\ldots x_m\in f(s_E^{-1}(v))$, the vertex $s_F(x_m)$ is regular. If it is not $0$-regular, then $g(x_m)=e_k\ldots e_i$ for some $k$. Now, there exists $r\in FP(G)$ such that $e_k\ldots e_{i-1}er=g(x)$ for $x\in s_F^{-1}(s_F(x_m))$. Also, there exists $r'\in FP(F)$ such that $x_1\ldots x_{m-1}xr'\in f(s^{-1}_E(v))$. Hence,
\begin{equation}
g(x_1\ldots x_{m-1}xr')=g(x_1\ldots x_{m-1})g(x)g(r')=e_1\ldots e_{k-1}e_k\ldots e_{i-1}erg(r'),
\end{equation}
so $e_1\ldots e_ierg(r')\in g(f(s^{-1}_E(v)))$, as needed. On the other hand, if $s_F(x_m)\in {\rm reg}_0(F)$, then either we argue as above, or $e_1\ldots e_i=g(x_1\ldots x_{m-1})$. Reasoning inductively, we arrive at $e_1\ldots e_i=g(x_1)$, which ends the argument because $s_F(x_1)=f(v)\notin {\rm reg}_0(F)$.

Finally, if $v\in\mathrm{reg}_0(E)$, then $f(s_E^{-1}(v))$ is either $f(v)$ or a loop at $f(v)$ such that every vertex of this loop emits only one edge.
In the former case we are done as $g\circ f$ satisfies the regularity condition at~$v$. The latter case splits into two subcases: either the loop at 
$f(v)$ has only one vertex or more than one vertex. In the former subcase, $f(v)\in\mathrm{reg}_0(F)$, so $g$ either maps $s_F^{-1}(f(v))$
to  $g(f(v))$ or to a loop at $g(f(v))$ such that every vertex of this loop emits only one edge. In the former instance, we are done because
\begin{equation}
g(f(s_E^{-1}(v)))=g((s_F^{-1}(f(v)))^n)=g(f(v)),
\end{equation}
 where $n\in\mathbb{N}\setminus\{0\}$. In the latter instance, we are also done as then
$g\circ f$ maps the loop  $s_E^{-1}(v)$ to a positive power of a loop at $g(f(v))$ such that every vertex of this loop emits only one edge,
which again is a loop whose all vertices emit only one edge, rendering the regularity condition satisfied. In the latter subcase, every vertex
in the loop at $f(v)$ is a regular vertex that is not 0-regular. Therefore, $g$ maps each edge of the loop at $f(v)$ to a positive-length path whose each vertex emits only one edge, so $g(f(s^{-1}_E(v)))$ is again a loop whose each vertex emits exactly one edge.
\end{proof}
\begin{definition}
The subcategory of the category $MIPG$ and $MBPG$ given by Lemma~\ref{rlemma} is denoted by $RMIPG$ and $RMBPG$.
\end{definition}

\section{Covariant functors into categories of algebras}

\subsection{Path algebras}

Let $SR$ denote the category of semirings with semiring homomorphisms and let $IPG$ stand for the category of graphs with path 
homomorphisms of graphs injective on vertices
 as morphisms. One can easily check that the
assignment
\begin{gather}
\mathrm{Obj}(IPG)\ni E\stackrel{R}{\longmapsto}R_E\in\mathrm{Obj}(SR),\nonumber\\
\mathrm{Mor}(IPG)\ni (f\colon E\to F)
\stackrel{R}{\longmapsto}(f_*\colon R_E\to R_F)\in\mathrm{Mor}(SR),\nonumber\\
R_E\ni A\stackrel{f_*}{\longmapsto}f(A)\in R_F\,,
\end{gather}
defines a covariant functor. As the image of a finite subset is a finite subset, the same assignment
but with $R$ replaced by $R^f$ also defines a covariant functor to the category $SR$. In what follows, we explore such covariant functoriality for different types of algebras given by graphs.

Let $KA$ be the category
of algebras over a field~$k$ with algebra homomorphisms as morphisms, and let $UKA$ be the category
of unital algebras over a field~$k$ with unital algebra homomorphisms as morphisms. 

\begin{proposition}\label{prop226}
{\rm (cf.} \cite[Section~2.2]{kr-g09}{\rm )} The assignment
\begin{align*}
\mathrm{Obj}(IPG)\ni E&\stackrel{}{\longmapsto}kE\in\mathrm{Obj}(KA),\\
\mathrm{Mor}(IPG)\ni (f\colon E\to F)
&\stackrel{}{\longmapsto}(f_*\colon kE\to kF)\in\mathrm{Mor}(KA),\\
\forall\; p\in FP(E)\colon kE\ni \chi_p&\stackrel{f_*}{\longmapsto}\chi_{f(p)}\in kF\,,
\end{align*}
defines a covariant functor. Furthermore, the same assignment restricted to the subcategory $BPG$
yields a covariant functor to the category~$UKA$.
\end{proposition}
\begin{proof}
For starters, $f_*\colon kE\to kF$ is, clearly, a linear map. To check that it is a homomorphism of algebras,
using the injectivity of $f\colon FP(E)\to FP(F)$ restricted to $E^0$, we compute:
\begin{align}
f_*(\chi_p\chi_q)&=\delta_{t_{PE}(p),s_{PE}(q)}f_*(\chi_{pq})=\delta_{t_{PE}(p),s_{PE}(q)}\chi_{f(pq)}=\delta_{t_{PE}(p),s_{PE}(q)}\chi_{f(p)f(q)}
\nonumber\\&=\delta_{t_{PF}(f(p)),s_{PF}(f(q))}\chi_{f(p)f(q)}=\chi_{f(p)}\chi_{f(q)}=f_*(\chi_p)f_*(\chi_q).
\end{align}
Hence, $f_*$ is an algebra homomorphism, as claimed. 
The covariance  is obvious because
\begin{equation}
(f\circ g)_*=f_*\circ g_*.
\end{equation}

 Finally, if $E$ and $F$ are graphs with finitely many vertices, then $kE$ and $kF$ are unital algebras
with $\sum_{v\in E^0}\chi_v$ and $\sum_{w\in F^0}\chi_{w}$ as respective units, and the unitality of~$f_*$
follows from the bijectivity of $f\colon FP(E)\to FP(F)$
 restricted to vertices:
\begin{equation}\label{sumvex}
f_*\left(\sum_{v\in E^0}\chi_v\right)=\sum_{v\in E^0}\chi_{f(v)}=\sum_{w\in F^0}\chi_{w}\,.
\vspace*{-6mm}\end{equation}
\end{proof}

\subsection{Cohn path algebras}

To pass to the functoriality of the construction of Cohn and Leavitt path algebras, first we need to establish the 
functoriality of the extended graph construction. We leave the routine proof of the following lemma to the reader.
\begin{lemma}\label{extlem}
Let $f\colon E\to F$ be any path homomorphism of graphs. Then the formulas
\begin{gather*}
\forall\;v\in E^0\colon \bar{f}(v):=f(v),\quad
\forall\;e\in E^1\colon \bar{f}(e):=f(e),\,\bar{f}(e^*):=f(e)^*,\\
\forall\;p:=x_1\ldots x_n\in FP(\bar{E}),\,x_1,\ldots,x_n\in \bar{E}^1\colon 
\bar{f}(p):=\bar{f}(x_1)\ldots\bar{f}(x_n),
\end{gather*}
define a path homomorphism $\bar{f}\colon\bar{E}\to \bar{F}$ of extended graphs such that the assignment 
$E\longmapsto \bar{E}$, $f\longmapsto \bar{f}$,
yields a 
covariat endofunctor of the category of graphs and path homomorphisms of graphs.
Furthermore, this endofunctor restricts to an endofunctor of the subcategory~$BPG$.
\end{lemma}

We are now ready to determine the covariant functoriality of assigning Cohn path algebras to graphs.
\begin{proposition}\label{cohncov}
The assignment
\begin{align*}
\mathrm{Obj}(MIPG)\ni E&\stackrel{}{\longmapsto}C_k(E)\in\mathrm{Obj}(KA),\nonumber\\
\mathrm{Mor}(MIPG)\ni (f\colon E\to F)
&\stackrel{}{\longmapsto}(f_*^C\colon C_k(E)\to C_k(F))\in\mathrm{Mor}(KA),\nonumber\\
\forall\; p\in FP(\bar{E})\colon C_k(E)\ni [\chi_p]&\stackrel{f_*^C}{\longmapsto}[\chi_{\bar{f}(p)}]\in C_k(F)\,,
\label{iipgfunctor}
\end{align*}
defines a covariant functor. Furthermore, the same assignment restricted to the subcategory  $MBPG$ 
yields a covariant functor to the category~$UKA$.
\end{proposition}
\begin{proof}
Combining Proposition~\ref{prop226} and Lemma~\ref{extlem} with the fact that $MIPG$ is a subcategory
of $IPG$, we conclude that the assignment 
\begin{align}
\mathrm{Obj}(MIPG)\ni E &\stackrel{}{\longmapsto}k\bar{E}\in\mathrm{Obj}(KA),\\ \nonumber
\mathrm{Mor}(MIPG)\ni (f\colon E\to F) &\stackrel{}{\longmapsto}(\bar{f}_*\colon k\bar{E}\to k\bar{F})\in\mathrm{Mor}(KA),\\
\forall\; p\in FP(\bar{E})\colon k\bar{E}\ni \chi_p &\stackrel{\bar{f}_*}{\longmapsto}\chi_{\bar{f}(p)}\in k\bar{F}\,,\nonumber
\end{align}
defines a covariant functor. Next, let us check that for any graph $E$ and $e\in E^1$ such that $f(e)$
is of positive length ($f(e)=:(f_1,\ldots,f_n)$) we have
\begin{align}
[\bar{f}_*(\chi_{e^*}\chi_{e})] &=[\chi_{\bar{f}(e^*e)}]
=
[\chi_{f(e)^*}][\chi_{f(e)}]
=
[\chi_{f_n^*}]\ldots[\chi_{f_1^*}][\chi_{f_1}]\ldots[\chi_{f_n}]
\\ \nonumber &=
[\chi_{f_n^*}]\ldots[\chi_{f_2^*}][\chi_{t_F(f_1)}][\chi_{f_2}]\ldots[\chi_{f_n}]
=
[\chi_{f_n^*}]\ldots[\chi_{f_2^*}][\chi_{f_2}]\ldots[\chi_{f_n}]
\\ \nonumber &=
[\chi_{t_F(f_n)}]
=
[\bar{f}_*(\chi_{t_E(e)})].
\end{align}
Note that, if $f(e)$ is a vertex, the above calculation is trivial. 

Assume now that $e,e'\in E^1$, $e\neq e'$ and $s_E(e)= s_E(e')$. 
Then, by the condition \eqref{mcondition}, we know that
$f(e)\not\preceq f(e')$ and $f(e')\not\preceq f(e)$. 
Hence, neither $f(e)$ nor $f(e')$ can be a vertex. Indeed, suppose that  $f(e)$ is a vertex. Then
\begin{equation}
f(e)=s_{PF}(f(e))=f(s_{PE}(e))=f(s_{PE}(e'))=s_{PF}(f(e')),
\end{equation}
so $f(e')=f(e)f(e')$ contradicting $f(e)\not\preceq f(e')$. The path $f(e')$ cannot be a vertex by the symmetric
argument. Consequently, both $f(e)$ and $f(e')$ are paths of positive length. 
Let $f(e)=:f_1\ldots f_n$ and $f(e')=:f'_1\ldots f'_{n'}$.
Now, the condition $f(e)\not\preceq f(e')$ and $f(e')\not\preceq f(e)$ implies that there exists an index
$i\leq\min\{n,n'\}$ such that $f_i\neq f'_i$.
It follows that
\begin{equation}
[\bar{f}_*(\chi_{e^*}\chi_{e'})] =[\chi_{\bar{f}(e^*e')}]
 =[\chi_{f(e)^*}][\chi_{f(e')}]=
[\chi_{f_n^*}]\ldots[\chi_{f_1^*}][\chi_{f'_1}]\ldots[\chi_{f'_{n'}}]=
0.
\end{equation}
Note that, if $s_E(e)\neq s_E(e')$, the above calculation is trivial.

Summarizng, we infer that $f_*^C$ is a~well defined
algebra homomorphism. The functoriality $(f\circ g)_*^C=f_*^C\circ g_*^C$ follows
immediately from the functoriality $\overline{(f\circ g)}_*=\bar{f}_*\circ \bar{g}_*$.
Finally, the unitality of $f^C_*$ for $f\in {\rm Mor}(MBPG)$ follows immediately from~\eqref{sumvex}. 
\end{proof}
\begin{example}\label{ex4.4}
{\rm Let 
$$
E:=\quad
\begin{tikzpicture}[auto,swap]
\tikzstyle{vertex}=[circle,fill=black,minimum size=3pt,inner sep=0pt]
\tikzstyle{edge}=[draw,->]
\tikzstyle{cycle1}=[draw,->,out=130, in=50, loop, distance=40pt]
\tikzstyle{cycle2}=[draw,->,out=130, in=50, loop, distance=70pt]
   
\node[vertex,label=below:$v$] (0) at (0,0) {};

\path (0) edge[cycle1] node[above] {$e$} (0);

\end{tikzpicture}\text{ and }\quad
F:=\quad
\begin{tikzpicture}[auto,swap]
\tikzstyle{vertex}=[circle,fill=black,minimum size=3pt,inner sep=0pt]
\tikzstyle{edge}=[draw,->]
\tikzstyle{cycle1}=[draw,->,out=130, in=50, loop, distance=40pt]
\tikzstyle{cycle2}=[draw,->,out=130, in=50, loop, distance=70pt]
   
\node[vertex,label=below:$v$] (0) at (0,0) {};


\end{tikzpicture}.
$$
Then $f\colon FP(E)\to FP(F)$ must be the constant map. It satisfies the assumptions
of Proposition~\ref{cohncov} because there is only one vertex and only one edge in~$E$. 
Hence, we have a unital algebra homomorphism
\begin{equation}
f_*^C\colon C_k(E)\ni [\chi_e ]\longmapsto f_*^C([\chi_e ]):=[\chi_{f(e)}]=1\in C_k(F).
\end{equation}
After identifying $C_k(E)$ with the polynomial Toeplitz algebra $k\langle s,s^*\rangle/\langle s^*s-1\rangle$
by mapping $[\chi_e]$ to the generating isometry $s$ and $[\chi_{e^*}]$ to its adjoint $s^*$,
and identifying $C_k(F)$ with~$k$, we see that $f_*^C$ becomes the map given by
$\mathrm{ev}_1(s)=1$ and $\mathrm{ev}_1(s^*)=1$.}
\end{example}
Since two edges are equal $e=e'$   if and only if $e\preceq e'$,
the condition \eqref{mcondition} is equivalent to
\begin{equation}\label{newm}
f(e)\preceq f(e')\quad\Rightarrow\quad e\preceq e'\,.
\end{equation}
The same condition written for vertices is equivalent to $f$ being injective on vertices. Therefore, it is
tempting to assume \eqref{newm} for all paths. However, this does not work in the above example because
\begin{equation}
f(e)=v\preceq v=f(v)\quad\text{and}\quad e\not\preceq v\,.
\end{equation}
On the other hand, \eqref{mcondition} implies the injectivity of $f$ on edges, so one might be tempted
to weaken \eqref{mcondition} to the injectivity of $f$ on edges. However, this does not work in the following
example.
\begin{example}
{\rm
Let 
$$
E:=\quad
\begin{tikzpicture}[auto,swap]
\tikzstyle{vertex}=[circle,fill=black,minimum size=3pt,inner sep=0pt]
\tikzstyle{edge}=[draw,->]
\tikzstyle{cycle1}=[draw,->,out=130, in=50, loop, distance=40pt]
\tikzstyle{cycle2}=[draw,->,out=130, in=50, loop, distance=70pt]
   
\node[vertex,label=below:$v$] (0) at (0,0) {};
\node (1) at (0,1.25) {};

\path (0) edge[cycle1] node {$e_1$} (0);
\path (0) edge[cycle2] node[above] {$e_2$} (0);

\end{tikzpicture}\text{ and }\quad
F:=\quad
\begin{tikzpicture}[auto,swap]
\tikzstyle{vertex}=[circle,fill=black,minimum size=3pt,inner sep=0pt]
\tikzstyle{edge}=[draw,->]
\tikzstyle{cycle1}=[draw,->,out=130, in=50, loop, distance=40pt]
\tikzstyle{cycle2}=[draw,->,out=130, in=50, loop, distance=70pt]
   
\node[vertex,label=below:$v$] (0) at (0,0) {};

\path (0) edge[cycle1] node[above] {$e$} (0);

\end{tikzpicture}.
$$
Then the path homomorphism of graphs given by $f(e_2):=e$ and $f(e_1):=e^2$ is injective on edges and vertices,
but the map $f_*^C\colon C_k(E)\to C_k(F)$ is not well defined because
\begin{equation}
0=f_*^C([\chi_{e_2^*}][\chi_{e_1}])=[\chi_{\bar{f}(e_2^*e_1)}]=[\chi_{f(e_2)^*}][\chi_{f(e_1)}]=
[\chi_{e^*}][\chi_{e^2}]=[\chi_{e}]\neq0.
\end{equation}
Consequently, by Proposition~\ref{cohncov},
the injectivity on edges does not imply~\eqref{mcondition}.}
\end{example}
\begin{example}
{\rm Let 
$$
E:=\quad
\begin{tikzpicture}[auto,swap]
\tikzstyle{vertex}=[circle,fill=black,minimum size=3pt,inner sep=0pt]
\tikzstyle{edge}=[draw,->]
\tikzstyle{cycle1}=[draw,->,out=130, in=50, loop, distance=40pt]
\tikzstyle{cycle2}=[draw,->,out=130, in=50, loop, distance=70pt]
   
\node[vertex,label=below:$v$] (0) at (0,0) {};
\node (1) at (0,1.25) {};

\path (0) edge[cycle1] node {$e_1$} (0);
\path (0) edge[cycle2] node[above] {$e_2$} (0);

\end{tikzpicture}\text{ and }\quad
F:=\quad
\begin{tikzpicture}[auto,swap]
\tikzstyle{vertex}=[circle,fill=black,minimum size=3pt,inner sep=0pt]
\tikzstyle{edge}=[draw,->]
\tikzstyle{cycle1}=[draw,->,out=130, in=50, loop, distance=40pt]
\tikzstyle{cycle2}=[draw,->,out=130, in=50, loop, distance=70pt]
   
\node[vertex,label=below:$v$] (0) at (0,0) {};


\end{tikzpicture}.
$$
Then \eqref{mcondition} clearly fails for the constant map.}
\end{example}
\begin{example}
{\rm Let \begin{equation*}
E:=\quad\begin{tikzpicture}[auto,swap]
\tikzstyle{vertex}=[circle,fill=black,minimum size=3pt,inner sep=0pt]
\tikzstyle{edge}=[draw,->]
\tikzstyle{cycle1}=[draw,->,out=130, in=50, loop, distance=40pt]
\tikzstyle{cycle2}=[draw,->,out=100, in=30, loop, distance=40pt]
   
\node[vertex,label=below:$v$] (0) at (0,0) {};
\node[vertex,label=below:$w$] (1) at (1,0) {};

\path (0) edge[edge, bend right] node[below] {$e_2$} (1);
\path (0) edge[edge,bend left] node[above] {$e_1$} (1);

\end{tikzpicture}\quad\text{and}\quad
F:=\quad\begin{tikzpicture}[auto,swap]
\tikzstyle{vertex}=[circle,fill=black,minimum size=3pt,inner sep=0pt]
\tikzstyle{edge}=[draw,->]
\tikzstyle{cycle1}=[draw,->,out=130, in=50, loop, distance=40pt]
\tikzstyle{cycle2}=[draw,->,out=100, in=30, loop, distance=40pt]
   
\node[vertex,label=below:$v$] (0) at (0,0) {};
\node[vertex,label=below:$w$] (1) at (1,0) {};

\path (0) edge[cycle1] node[above] {$e$} (0);
\path (0) edge[edge] node[above] {$g$} (1);

\end{tikzpicture}.
\end{equation*}
Then the path homomorphism given by identity on vertices and by the formulas
$f(e_1):=g$ and $f(e_2):=eg$ is a morphism in $MIPG$, but
\begin{equation}
\bar{f}(e_1^*)=g^*\preceq g^*e^*=\bar{f}(e_2^*)\quad\text{and}\quad e_1^*\neq e_2^*\,,
\end{equation}
so $\bar{f}$ is not  a morphism in~$MIPG$.}
\end{example}

\subsection{Leavitt path algebras and graph C*-algebras}

Finally, we consider the covariant functoriality of the construction of the Leavitt path algebra.
\begin{theorem}\label{levcov}
The assignment
\begin{align*}
\mathrm{Obj}(RMIPG)\ni E &\stackrel{}{\longmapsto}L_k(E)\in\mathrm{Obj}(KA),\nonumber\\
\mathrm{Mor}(RMIPG)\ni (f\colon E\to F)
&\stackrel{}{\longmapsto}(f_*^L\colon L_k(E)\to L_k(F))\in\mathrm{Mor}(KA),\nonumber\\
\forall\; p\in FP(\bar{E})\colon L_k(E)\ni [\chi_p]&\stackrel{f_*^L}{\longmapsto}[\chi_{\bar{f}(p)}]\in L_k(F)\,,
\end{align*}
defines a covariant functor. 
Furthermore, the same assignment restricted to the subcategory $RMBPG$
yields a covariant functor to the category~$UKA$.
\end{theorem}
\begin{proof}
Since one can view the Leavitt path algebras as quotients of Cohn path algebras, we infer from Proposition~\ref{cohncov} that it suffices to show that
\begin{equation}\label{reg}
\forall\;v\in\mathrm{reg}(E)\colon \sum_{e\in s_E^{-1}(v)}[\chi_{\bar{f}(ee^*)}]=[\chi_{\bar{f}(v)}].
\end{equation}
To this end, we need to use the regularity of $f$. For starters, let us observe that \eqref{reg} is easily satisfied
for any $v\in\mathrm{reg}_0(E)$. Indeed, if $f(e)=w\in F^0$, then 
\begin{equation}
[\chi_{\bar{f}(ee^*)}]=[\chi_w]=[\chi_{\bar{f}(v)}].
\end{equation}
If $f(e)=:p_1\ldots p_n$, $p_1\,,\ldots\,, p_n\in F^1$, is a loop whose all vertices emit exactly one edge, then we obtain
\begin{equation}
[\chi_{\bar{f}(ee^*)}]=[\chi_{p_1}]\ldots[\chi_{p_n}][\chi_{p_n^*}]\ldots[\chi_{p_1^*}] =[\chi_{\bar{f}(v)}].
\end{equation}

Now, let us consider $v\in\mathrm{reg}(E)\setminus\mathrm{reg}_0(E)$.
Then, by the bijectivity assumption, we have
\begin{equation}\label{2eq}
\sum_{e\in s_E^{-1}(v)}[\chi_{\bar{f}(ee^*)}]
=\sum_{e\in s_E^{-1}(v)}[\chi_{f(e)}][\chi_{f(e)^*}]
=\sum_{p_i\in f(s^{-1}_E(v))}[\chi_{p_i}][\chi_{p_i^*}].
\end{equation}
Next, let $n$ be the length of a longest path in~$f(s^{-1}_E(v))$. Then we can partition the set $f(s^{-1}_E(v))$ by path lengths
to obtain
\begin{align*}
\sum_{p_i\in f(s^{-1}_E(v))}[\chi_{p_i}][\chi_{p_i^*}]
&=
\sum_{j=1}^n\;\sum_{p_{i,j}\in f(s^{-1}_E(v))\cap FP_j(F)}[\chi_{p_{i,j}}][\chi_{p_{i,j}^*}]
\\ &=
\sum_{j=1}^{n-1}\left(\sum_{p_{i,j}\in f(s^{-1}_E(v))\cap FP_j(F)}[\chi_{p_{i,j}}][\chi_{p_{i,j}^*}]\right)
+\sum_{p_{i,n}\in f(s^{-1}_E(v))\cap FP_n(F)}[\chi_{p_{i,n}}][\chi_{p_{i,n}^*}].
\end{align*}
Writing $p_{i,n}=:q_{i,n-1}x_{i,n}$, where $x_{i,n}\in F^1$, we obtain
\begin{align}
\sum_{p_{i,n}\in f(s^{-1}_E(v))\cap FP_n(F)}[\chi_{p_{i,n}}][\chi_{p_{i,n}^*}]
&=
\sum_{q_{i,n-1}\in E_{f(v)}^{n-1}}\;\sum_{x_{i,n}\in s_F^{-1}(t_F(q_{i,n-1}))}
[\chi_{q_{i,n-1}}][\chi_{x_{i,n}}][\chi_{x_{i,n}^*}][\chi_{q_{i,n-1}^*}]
\\ \nonumber &=
\sum_{q_{i,n-1}\in E_{f(v)}^{n-1}}
[\chi_{q_{i,n-1}}][\chi_{q_{i,n-1}^*}].
\end{align}
Here $E_{f(v)}^{n-1}$ stands for all paths of length $n-1$ that where extended to paths of length $n$ 
in the construction of~$f(s^{-1}_E(v))$, and the first step follows from Definition~\ref{regularity}(1). Furthermore,
combining the sets $f(s^{-1}_E(v))\cap FP_{n-1}(F)$ and $E_{f(v)}^{n-1}$ we obtain all paths of length $n-1$ that 
are given by extending  paths of length $n-2$ by all possible edges.
Hence, using an analogous notation, we infer that
\begin{gather}
\sum_{p_{i,n-1}\in f(s^{-1}_E(v))\cap FP_{n-1}(F)}[\chi_{p_{i,n-1}}][\chi_{p_{i,n-1}^*}]
+\sum_{q_{i,n-1}\in E_{f(v)}^{n-1}}
[\chi_{q_{i,n-1}}][\chi_{q_{i,n-1}^*}]
\\ \nonumber
=\sum_{q_{i,n-2}\in E_{f(v)}^{n-2}}\;\sum_{x_{i,n-1}\in s_F^{-1}(t_F(q_{i,n-2}))}
[\chi_{q_{i,n-2}}][\chi_{x_{i,n-1}}][\chi_{x_{i,n-1}^*}][\chi_{q_{i,n-2}^*}]
\\ \nonumber
=\sum_{q_{i,n-2}\in E_{f(v)}^{n-2}}
[\chi_{q_{i,n-2}}][\chi_{q_{i,n-2}^*}].
\end{gather}
Iterating this procedure, we conclude that
\begin{align}
\sum_{p_i\in f(s^{-1}_E(v))}[\chi_{p_i}][\chi_{p_i^*}]&=
\sum_{p_{i,1}\in f(s^{-1}_E(v))\cap FP_1(F)}[\chi_{p_{i,1}}][\chi_{p_{i,1}^*}]+\sum_{q_{i,1}\in E^1_{f(v)}}[\chi_{q_{i,1}}][\chi_{q_{i,1}^*}]
\\ \nonumber
&=\sum_{e\in s_F^{-1}(f(v))}[\chi_{e}]
[\chi_{e^*}]=[\chi_{\bar{f}(v)}].
\end{align}
Combining this with \eqref{2eq}, we complete the proof.
\end{proof}

Let us end this section by extending Theorem~\ref{levcov} to graph C*-algebras. For starters, recall that for any
$*$-homomorphism 
$\phi:L_\mathbb{C}(E)\to L_\mathbb{C}(F)$ there exists a unique $*$-homomorphism
 \mbox{$\widetilde{\phi}:C^*(E)\to C^*(F)$} extending $\phi$ (e.g., see~\cite[Theorem~4.4]{at-11}). Now, we can claim:
\begin{corollary}\label{c*cov}
Let $C^*\!A$ denote the category of C*-algebras and $*$-homomorphisms. The assignment
\begin{align*}
\mathrm{Obj}(RMIPG)\ni E&\stackrel{}{\longmapsto}C^*(E)\in\mathrm{Obj}(C^*\!A),\nonumber\\
\mathrm{Mor}(RMIPG)\ni (f\colon E\to F)
&\stackrel{}{\longmapsto}(\widetilde{f_*^L}\colon C^*(E)\to C^*(F))\in\mathrm{Mor}(C^*\!A),\nonumber\\
\forall\; v\in E^0\colon C^*(E)\ni P_v&\stackrel{\widetilde{f_*^L}}{\longmapsto}P_{\bar{f}(v)}\in C^*(F)\,,\\
\forall\; e\in E^1\colon C^*(E)\ni S_e&\stackrel{\widetilde{f_*^L}}{\longmapsto}S_{\bar{f}(e)}\in C^*(F)\,,
\end{align*}
where $f_*^L$ is given by Theorem~\ref{levcov} (for $k=\mathbb{C}$), defines a covariant functor. 
Furthermore, the same assignment restricted to the subcategory  $RMBPG$
yields a covariant functor to the category of unital C*-algebras and unital $*$-homomorphisms.
\end{corollary}

\begin{example}{\rm Let
\begin{equation*}
E:=\quad\begin{tikzpicture}[auto,swap]
\tikzstyle{vertex}=[circle,fill=black,minimum size=3pt,inner sep=0pt]
\tikzstyle{edge}=[draw,->]
\tikzstyle{cycle1}=[draw,->,out=130, in=50, loop, distance=40pt]
\tikzstyle{cycle2}=[draw,->,out=100, in=30, loop, distance=40pt]
\node[vertex,label=below:$v$] (0) at (0,0) {};
\node[vertex,label=below:$w$] (1) at (1,0) {};

\path (0) edge[edge] node[above] {$e$} (1);

\end{tikzpicture}\quad\text{and}\quad
F:=\quad\begin{tikzpicture}[auto,swap]
\tikzstyle{vertex}=[circle,fill=black,minimum size=3pt,inner sep=0pt]
\tikzstyle{edge}=[draw,->]
\tikzstyle{cycle1}=[draw,->,out=130, in=50, loop, distance=40pt]
\tikzstyle{cycle2}=[draw,->,out=100, in=30, loop, distance=40pt]
   
\node[vertex,label=below:$v$] (0) at (0,0) {};
\node[vertex,label=below:$w$] (1) at (1,0) {};

\path (0) edge[edge, bend right] node[below] {$e_1$} (1);
\path (1) edge[edge,bend right] node[above] {$e_2$} (0);

\end{tikzpicture}.
\end{equation*}
Then, the path homomorphism of graphs given by $f(e):=e_1$ is clearly a morphism in the category RMBPG. One can easily generalize this example to a path 
homomorphism $f_n:E_{n-1}\to F_n$, $n\geq 3$, from a graph given by a straight path of length $n-1$ to a graph obtained by closing this path to a loop by 
adding an edge from the last to the first vertex. Furthermore, using the standard identifications $C^*(E_{n-1})\cong M_n(\mathbb{C})$ and 
$C^*(F_n)\cong M_n(\mathbb{C})\otimes C(S^1)$~(e.g., see~\cite[Example~2.14]{i-r05}), we infer that $f_n$ corresponds to the tensorial inclusion 
$M_n(\mathbb{C})\hookrightarrow M_n(\mathbb{C})\otimes C(S^1):x\mapsto x\otimes 1$ via Corollary~\ref{c*cov}.}
\end{example}

\section{Mixed-pullback theorem}
\noindent
In \cite[Theorem~3.4]{cht21}, the authors together with A.~Chirvasitu proved a  pullback theorem for graph C*-algebras (cf. pushout-to-pullback theorem~\cite{hrt20,ht22}). The theorem was motivated by 
examples coming from noncommutative topology such as the quantum CW complex structure of the standard Podle\'s quantum  sphere. Since the result 
involves both covariant and contravariant functors from certain categories of graphs, we call it a mixed-pullback theorem. Herein, using 
Corollary~\ref{c*cov}, we substantially
improve this result by allowing loops, and restricting the monotonicity condition from arbitrary finite paths to edges  (see~\eqref{mcondition}).
Thus, the quantum CW-complex structure of the quantum real projective plane (Example~\ref{qrp2example}) and the basic Example~\ref{ex4.4} are now 
within the scope of the theorem. To simplify notation, in this section, given a morphism $f:FP(E)\to FP(F)$ in the category {\rm RMIPG}, we write $f_*$ for the induced C*-algebra homomorphism given by Corollary~\ref{c*cov}.

First, we recall the notions of a saturated subset and an admissible inclusion. A subset $H\subseteq E^0$ is called {\em saturated} if there does not exist a 
regular vertex $v\in E^0\setminus H$ such that $t_E(s^{-1}_E(v))\subseteq H$. An injective graph homomorphism $(\pi^0,\pi^1):F\hookrightarrow E$ is called 
an {\em admissible inclusion} (cf.~\cite[Definition~3.1]{cht21}) if the following conditions are satisfied:
\begin{enumerate}
\item[(A1)] $E^0\setminus \pi^0(F^0)$ is saturated,
\item[(A2)] $t_E^{-1}(\pi^0(F^0))\subseteq \pi^1(F^1)$.
\end{enumerate}
\begin{remark}
{\rm Note that in~\cite[Definition~3.1]{cht21} it is assumed that $E^0\setminus \pi^0(F^0)$ is hereditary and that we have the equality rather than the 
inclusion of sets in the condition (A2). However,  the reverse inclusion to the one in (A2) is automatically true for any graph homomorphism, and one can 
show that (A2) implies that $E^0\setminus \pi^0(F^0)$ is hereditary (e.g., see~\cite{ht22}).}
\end{remark}

Although the concept of admissibility allows the existence of breaking vertices, we need to control their behavior to formulate Theorem~\ref{main}.
Recall that, for $H\subseteq E^0$, we say that $v\in E^0\setminus H$ is a {\em breaking vertex} of $H$ whenever
\begin{equation}
|s^{-1}_E(v)|=\infty\qquad\text{and}\qquad 0<|s^{-1}_E(v)\cap t^{-1}_E(E^0\setminus H)|<\infty.
\end{equation}
By $B_H$ we denote the set of all breaking vertices of~$H$.
Note that a breaking vertex of $H$ becomes regular in the subgraph obtained by removing all vertices
in $H$ and all edges ending in~$H$.

Next, recall that, by~\cite[Corollary~3.5]{bhrsz-02},
if $(\pi^0,\pi^1):F\hookrightarrow E$ is an admissible inclusion, then there exists a surjective gauge-equivariant \mbox{$*$-homo}\-morphism $\pi^*:C^*(E)
\to C^*(F)$ given by 
\begin{equation}\label{admhom}
\pi^*(P_v)=\begin{cases} P_v & v\in\pi^0(F^0),\\0 & v\notin\pi^0(F^0),\end{cases}\quad v\in E^0,\qquad \pi^*(S_e)=\begin{cases} S_e & 
e\in\pi^1(F^1),\\ 0 & e\notin\pi^1(F^1),\end{cases}\quad e\in E^1,
\end{equation}
and whose kernel can be explicitely written as follows:
\begin{align}\label{kerpi}
\ker(\pi^*)&=
\Big{\langle}\{P_v~|~v\in E^0\setminus \pi^0(F^0)\}\cup\Big{\{}P_w-\hspace*{-12mm}\sum_{e\in \pi^1\big((\pi^0\circ s_F)^{-1}(w)\big)}
\hspace*{-12mm}S_eS_e^* \;\Big{|}\;
w\in B_{E^0\setminus \pi^0(F^0)}\Big{\}}\Big{\rangle}
\nonumber\\ &=
\overline{\mathrm{span}}\left\{S_\alpha S_\beta^*,\;S_\mu\Big(P_w-\hspace*{-12mm}\sum_{e\in \pi^1\big((\pi^0\circ s_F)^{-1}(w)\big)}
\hspace*{-12mm}S_eS_e^*\Big)S_\nu^*~\Big|~
\begin{matrix}
\alpha,\,\beta,\,\mu,\,\nu\in FP(E),\\
t(\alpha)=t(\beta) \in E^0\setminus \pi^0(F^0),  \\
t(\mu)=t(\nu)=w \in B_{E^0\setminus \pi^0(F^0)}\\
\end{matrix}
\right\}.
\end{align}

Finally, recall that a {\em vertex-simple loop} is a loop whose each vertex emits only one edge belonging to the loop. We say that a loop has an \emph{exit}
iff one of its vertices emits an edge not belonging to the loop.
We are now ready to prove the main result of this section,
 which generalizes~\cite[Theorem~3.4]{cht21}.
\begin{theorem}\label{main}
Let $(\pi^0_i,\pi^1_i):F_i\hookrightarrow E_i$, $i=1,2$, be admissible inclusions of graphs. Also, assume that $E_1$ is a graph whose every vertex-simple 
loop has an exit. Furthermore, let  \mbox{$f\colon E_1\to E_2$} be a morphism in the category $RMIPG$ mapping only breaking vertices to breaking vertices:
\begin{equation}\label{break}
f^{-1}(B_{E_2^0\setminus\pi_2^0(F_2^0)})\subseteq B_{E_1^0\setminus\pi_1^0(F_1^0)}
\end{equation}
and such that
\begin{equation}\label{comm}
f^{-1}(\pi_2^0(F_2^0))\subseteq \pi_1^0(F_1^0).
\end{equation}
Finally, assume that $f$ restricts to a morphism $f|:F_1\to F_2$ in the same category such that
\begin{equation}\label{regular}
f|\big((\pi^{0}_{1}\circ s_{F_1})^{-1}(B_{E_1^0\setminus\pi_1^0(F_1^0)})\big)\subseteq F^{1}_{2}\,.
\end{equation}
Then, if all the finite paths ending in $(E_2^0\setminus \pi^0_2(F_2^0))\cup B_{E_2^0\setminus\pi_2^0(F_2^0)}$ are in the image of $f$, 
the induced \mbox{$*$-homo}\-morphisms exist and render the diagram
\begin{equation}\label{pulldiag}
\xymatrix{
&
C^*(E_1) \ar[dl]_{\pi_1^*} \ar[dr]^{f_{*}}
&\\
C^*(F_1) \ar[dr]_{f|_{*}}
& & 
C^*(E_2) \ar[dl]^{\pi_2^*}
\\
&
C^*(F_2)
&
}
\end{equation}
a pullback diagram in the category of C*-algebras and $*$-homomorphisms.
\end{theorem}
\begin{proof}
Since $(\pi_i^0,\pi^1_i):F_i\hookrightarrow E_i$, $i=1,2$, are admissible, the induced $*$-homomorphisms $\pi^*_i$, $i=1,2$, exist and are given 
by~\eqref{admhom}. Next, since $f$ and $f|$ are morphisms in the category $RMIPG$, the induced $*$-homomorphisms $f_*$ and $f|_*$ exist by 
Corollary~\ref{c*cov}. 

Let us check that~\eqref{pulldiag} is a~commutative diagram. Pick a vertex \mbox{$v\in f^{-1}(\pi_2^0(F_2^0))\subseteq E_1^0$}. It follows that
$\pi_2^*(f_*(P_v))=\pi_2^*(P_{f(v)})=P_{f(v)}$. Now, we conclude from  \eqref{comm} that   $\pi_1^*(P_v)=P_v$ and, 
consequently, $f|_*(\pi_1^*(P_v))=f|_*(P_v)=P_{f(v)}$. On the other hand, if $v\notin f^{-1}(\pi_2^0(F_2^0))$, then $\pi_2^*(f_*(P_v))=\pi_2^*(P_{f(v)})=0$. 
Furthermore, since $f$ restricts 
to $f|$, we infer the inclusion $f(\pi_1^0(F_1^0))\subseteq \pi_2^0(F_2^0)$. Hence,
$\pi_1^0(F^0_1)\subseteq  f^{-1}(\pi_2^0(F_2^0))$, so $\pi_1^*(P_v)=0$. Consequently, 
\begin{equation}
(f|_*\circ\pi_1^*)(P_v)=0=(\pi_2^*\circ f_*)(P_v).
\end{equation}
An analogous argument shows that $\pi_2^*\circ f_*$ agrees with $f|_*\circ\pi_1^*$ on the elements $\{S_e\}_{e\in E^1}$.

To prove that~\eqref{pulldiag} is a pullback diagram of C*-algebras, we check the conditions of~\cite[Proposition~3.1]{gk-p99}. First, observe that, since 
every vertex-simple loop of $E_1$ has an exit and $f_*(P_v)\neq 0$ for all $v\in E_1^0$, the $*$-homomorphism $f_*$ is injective 
(e.g., see~\cite[Corollary~1.3]{w-sz02}).
Combining this with surjectivity of $\pi^*_1$ and $\pi^*_2$ and the commutativity of the diagram~\eqref{pulldiag},
 we infer that it suffices to check whether
\begin{equation}
\ker(\pi_2^*)\subseteq f_*(\ker(\pi_1^*)).
\end{equation}
Next, since $f_*(\ker(\pi_1^*))$ is a closed vector space, we can reduce proving the above inclusion to showing that all spanning elements of both types 
appearing in the second line of \eqref{kerpi} are in~$f_*(\ker(\pi_1^*))$.

By assumption,  finite paths $\alpha$ and $\beta$ with $t_{PE_2}(\alpha)=t_{PE_2}(\beta)\in E_2^0\setminus \pi^0_2(F_2^0)$ 
are in the image of~$f$,
so we can write $\alpha\beta^*=f\big(\widetilde\alpha \widetilde\beta^*\big)$
and $S_\alpha S_\beta^*=f_*\big(S_{\widetilde\alpha \widetilde\beta^*}\big)$. Furthermore, as 
\begin{equation}
f(\pi_1^0(F_1^0))\subseteq \pi_2^0(F_2^0),
\end{equation}
we infer that $f|_*(\pi_1^*(P_{t_{PE_1}(\widetilde\alpha)}))=0$ by the commutativity of the diagram. 
However, since $f|_*$ is injective on vertex projections, we conclude that
$t_{PE_1}(\widetilde{\alpha})\notin \pi^0_1(F^0_1)$, so $\pi_1^*(S_{\widetilde{\alpha}})=0$.
Hence, $S_{\widetilde\alpha}\in \ker\pi_1^*$, as needed.

Also by assumption,  finite paths $\mu$ and $\nu$ with $t_{PE_2}(\mu)=t_{PE_2}(\nu)=w\in B_{E_2^0\setminus\pi_2^0(F_2^0)}$ 
are in the image of~$f$. In particular, there exists $u\in E^0_1$ such that $f(u)=w$. It follows from \eqref{break} that 
$u\in B_{E_1^0\setminus\pi_1^0(F_1^0)}$. 
The commutativity of the diagram combined with the regularity assumption for $f|$ and \eqref{regular} allows us to compute:
\begin{equation}
f_*\Big(P_u-\hspace*{-8mm}\sum_{y\in \pi^1_1\big((\pi^0_1\circ s_{F_1})^{-1}(u)\big)}
\hspace*{-8mm}S_yS_y^*\Big)\quad=\quad P_w-\hspace*{-10mm}
\sum_{x\in \pi^1_2\big((\pi^0_2\circ s_{F_2})^{-1}(w)\big)}
\hspace*{-8mm}S_xS_x^*
\quad\in C^*(E_2).
\end{equation}
Therefore, for finite paths $\widetilde\mu$ and $\widetilde\nu$ such that $f(\widetilde\mu)=\mu$ and $f(\widetilde\nu)=\nu$, we obtain
\begin{equation}
f_*\Big(S_{\widetilde\mu}\big(P_u-\hspace*{-8mm}\sum_{y\in \pi^1_1\big((\pi^0_1\circ s_{F_1})^{-1}(u)\big)}
\hspace*{-8mm}S_yS_y^*\big)S_{\widetilde\nu}^*\Big)\quad=\quad S_\mu\big(P_w-\hspace*{-10mm}
\sum_{x\in \pi^1_2\big((\pi^0_2\circ s_{F_2})^{-1}(w)\big)}
\hspace*{-8mm}S_xS_x^*\big)S_\nu^*
\end{equation}
and
\begin{equation}
\pi_1^*\Big(S_{\widetilde\mu}\big(P_u-\hspace*{-8mm}\sum_{y\in \pi^1_1\big((\pi^0_1\circ s_{F_1})^{-1}(u)\big)}
\hspace*{-8mm}S_yS_y^*\big)S_{\widetilde\nu}^*\Big)\;=\;
\pi_1^*(S_{\widetilde\mu})\pi_1^*\big(P_u-\hspace*{-8mm}\sum_{y\in \pi^1_1\big((\pi^0_1\circ s_{F_1})^{-1}(u)\big)}
\hspace*{-8mm}S_yS_y^*\big)\pi_1^*(S_{\widetilde\nu}^*)=0,
\end{equation}
as required. Here the last step follows from the regularity of $(\pi^0_1)^{-1}(u)$ and the second
 Cuntz--Krieger relation.
\end{proof}
\begin{example}\label{qrp2example}{\rm
Let $q\in [0,1]$. The C*-algebra $C(\mathbb{R}{\rm P}_q^2)$ \cite{hrs-03} 
of the $q$-deformed real projective plane 
$\mathbb{R}{\rm P}^2_q$ \cite{h-pm96} can be identified with the following subalgebra of the Toeplitz algebra 
$\mathcal{T}$:
\begin{equation}
\{t\in\mathcal{T}~|~\sigma(t)(-x)=\sigma(t)(x)~\text{for all $x\in S^1$}\}.
\end{equation}
Here $\sigma:\mathcal{T}\to C(S^1)$ is the symbol map. We define  graphs $E_1$ and $E_2$ as follows:
\begin{equation*}\vspace*{-2mm}
\begin{tikzpicture}[auto,swap]
\tikzstyle{vertex}=[circle,fill=black,minimum size=3pt,inner sep=0pt]
\tikzstyle{edge}=[draw,->]
\tikzstyle{cycle1}=[draw,->,out=130, in=50, loop, distance=40pt]
\tikzstyle{cycle2}=[draw,->,out=100, in=30, loop, distance=40pt]
   
\node[vertex, label=below:$v~$] (0) at (0,0) {};
\node[vertex, label=below:$~w$] (3) at (1,0) {}; 

\path (0) edge[cycle1] node[above] {$s$} (0);
\path (0) edge[edge,bend left] node[above] {~$r$} (3);
\path (0) edge[edge,bend right] node[below] {$t$} (3);
\end{tikzpicture}
\qquad \qquad 
\begin{tikzpicture}[auto,swap]
\tikzstyle{vertex}=[circle,fill=black,minimum size=3pt,inner sep=0pt]
\tikzstyle{edge}=[draw,->]
\tikzstyle{cycle1}=[draw,->,out=130, in=50, loop, distance=40pt]
\tikzstyle{cycle2}=[draw,->,out=100, in=30, loop, distance=40pt]
   
\node[vertex,label=below:$v$] (0) at (0,0) {};
\node[vertex,label=below:$w$] (1) at (1,0) {};

\path (0) edge[cycle1] node[above] {$e$} (0);
\path (0) edge[edge] node[above] {~$f$} (1);
\end{tikzpicture}
\end{equation*}
\begin{equation*}
\text{\small Figure: Graph $E_1$ (on the left) and graph $E_2$ (on the right).}\vspace*{2mm}
\end{equation*}
Both $C(\mathbb{R}{\rm P}_q^2)$ and $\mathcal{T}$ can be realized as graph C*-algebras of the graphs $E_1$ 
and $E_2$, respectively. (See \cite{hs02} for the graph-algebraic presentation of both complex and real
quantum projective spaces.) Now,
consider the $*$-homomorphism covariantly induced by the path homomorphism 
$\varphi:{\rm FP}(E_1)\to {\rm FP}(E_2)$ given by the assignment
\begin{equation}
v\longmapsto v,\quad w\longmapsto w,\quad s\longmapsto e^2,\quad r\longmapsto f,\qquad t\longmapsto ef,
\end{equation}
which exists by Corollary~\ref{c*cov}.
Then, by Theorem~\ref{main}, the diagram
\begin{equation}\label{qrp2}
\xymatrix{
&
C(\mathbb{R}{\rm P}_q^2) \ar[dl]_{\pi^*} \ar[dr]^{\varphi_{*}}
&\\
C(\mathbb{R}{\rm P}^1) \ar[dr]_{\varphi|_{*}}
& & 
\mathcal{T} \ar[dl]^{\sigma}
\\
&
C(S^1)
&
}
\end{equation}
is a pullback diagram in the category of C*-algebras and $*$-homomorphisms. Here $\pi:=(\pi^0,\pi^1)$
is an admissible inclusion of the graph with one edge and one vertex into $E_1$. 
Furthermore,
 one immediately sees that all $*$-homo\-mor\-phisms are unital.
Observe that this pullback diagram yields a quantum
CW-complex structure (in the sense of \cite{dhmsz-20}) of the $q$-deformed real projective plane.
Better still, due to recent work of Atul Gothe, Corollary~\ref{c*cov} and Theorem~\ref{main}
can be applied much in the same way to all even quantum real projective spaces to unravel their
quantum CW-complex structure. (The odd case is slightly more difficult, but there is work in progress.)
}
\end{example}

\section*{Acknowledgements} 
\noindent
We are happy to thank Alexander Frei for his inspiration concerning the regularity condition, Jack Spielberg for 
sharing his insights regarding the covariant functoriality of graph C*-algebras, and Gilles G. de Castro for pointing to us the condition~\eqref{comm}. Also, Mariusz Tobolski
appreciates the funding of his visits to the University of Warsaw provided by the Thematic Research Programme
``Quantum Symmetries''. Morevover, this work is part of the project ``Applications of graph algebras and higher-rank 
graph algebras in noncommutative geometry'' partially supported by NCN grant UMO-2021/41/B/ST1/03387.


\begin{thebibliography}{999}

\bibitem{aasm17}
G.~Abrams, P.~Ara, M.~Siles Molina.
Leavitt path algebras. 
Lecture Notes in Mathematics, 2191. {\em Springer}, {\em London}, 2017. xiii+287 pp.

\bibitem{at-11}
G.~Abrams, M.~Tomforde. 
Isomorphism and Morita equivalence of graph algebras. 
{\em Trans. Amer. Math. Soc.} 363 (2011), no. 7, 3733--3767.

\bibitem{ass-06}
I.~Assem, D.~Simson, A.~Skowro\'nski. 
Elements of the representation theory of associative algebras. Vol. 1. 
Techniques of representation theory. 
London Mathematical Society Student Texts, 65. {\em Cambridge University Press}, {\em Cambridge}, 2006. x+458 pp.

\bibitem{bhrsz-02}
T.~Bates, J-H.~Hong, I.~Raeburn, W.~Szyma\'nski. 
The ideal structure of the $C^*$-algebras of infinite graphs. 
{\em Illinois J. Math.} 46 (2002), no. 4, 1159--1176.

\bibitem{bprs-00}
T.~Bates, D.~Pask, I.~Raeburn, W.~Szyma\'nski. 
The C*-algebras of row-finite graphs. 
{\em New York J. Math.} 6 (2000), 307--324.

\bibitem{c-pm66}
P. M. Cohn. Some remarks on the invariant basis property.
{\em Topology}, 5 (1966), 21--228.

\bibitem{cht21}
A.~Chirvasitu, P.~M.~Hajac, M.~Tobolski.
Non-surjective pullbacks of graph C*-algebras from non-injective pushouts of graphs. 
{\em Bull. Lond. Math. Soc.} 53 (2021), no. 1, 1--15.

\bibitem{ck-80}
J.~Cuntz, W.~Krieger. 
A class of C*-algebras and topological Markov chains. 
{\em Invent. Math.} 56 (1980), no. 3, 251--268.

\bibitem{dhmsz-20}
F.~D'Andrea, P.~M.~Hajac, T.~Maszczyk, A.~Sheu, B.~Zieli\'nski.
The K-theory type of quantum CW-complexes.
arXiv:2002.09015.

\bibitem{flr-00}
N.~J.~Fowler, M.~Laca, I.~Raeburn. 
The C*-algebras of infinite graphs. 
{\em Proc. Amer. Math. Soc.} 128 (2000), no. 8, 2319--2327.

\bibitem{kr-g09}
K.~R.~Goodearl. 
Leavitt path algebras and direct limits. 
{\em Rings, modules and representations}, 165--187, Contemp. Math., 480, 
{\em Amer. Math. Soc.}, {\em Providence}, {\em RI}, 2009.

\bibitem{h-pm96}
P.~M.~Hajac. 
Strong connections on quantum principal bundles. 
{\em Comm. Math. Phys.} 182 (1996), no. 3, 579--617.

\bibitem{hrs-03}
P.~M.~Hajac, R.~Matthes, W.~Szyma\'nski. 
Quantum real projective space, disc and spheres. 
{\em Algebr. Represent. Theory} 6 (2003), no. 2, 169--192

\bibitem{hrt20}
P.~M.~Hajac, S.~Reznikoff, M.~Tobolski.
Pullbacks of graph C*-algebras from admissible pushouts of graphs.
{\em Quantum dynamics--dedicated to Professor Paul Baum}, 169--178, 
Banach Center Publ., 120, {\em Polish Acad. Sci. Inst. Math.}, Warsaw, 2020.

\bibitem{ht22}
P.~M.~Hajac, M.~Tobolski.
From length-preserving pushouts of graphs to one-surjective pullbacks of graph algebras.
To appear in {\em J. Noncommut. Geom.}

\bibitem{hs02}
J.~H.~Hong, W.~Szyma\'nski. 
Quantum spheres and projective spaces as graph algebras. 
{\em Comm. Math. Phys.} 232 (2002), no. 1, 157--188.

\bibitem{gk-p99}
G.~K.~Pedersen. 
Pullback and pushout constructions in C*-algebra theory. 
{\em J. Funct. Anal.} 167 (1999), no. 2, 243--344.

\bibitem{i-r05}
I.~Raeburn. 
{\em Graph algebras}. 
CBMS Regional Conference Series in Mathematics, 103. 
AMS, Providence, RI, 2005. vi+113 pp.

\bibitem{s-j02}
J.~Spielberg. 
A functorial approach to the C*-algebras of a graph. 
{\em Internat. J. Math.} 13 (2002), no. 3, 245--277.

\bibitem{s-jr83}
J. R. Stallings.
Topology of finite graphs. 
{\em Invent. Math.} 71 (1983), no. 3, 551--565.

\bibitem{w-sz02}
W.~Szyma\'nski.
General Cuntz-Krieger uniqueness theorem. 
{\em Internat. J. Math.} 13 (2002), no. 5, 549--555.

\end{thebibliography}
\end{document}